\documentclass[a4paper,10pt,reqno]{article}

\usepackage[utf8]{inputenc}
\usepackage[T1]{fontenc}
\usepackage{lmodern}
\usepackage[english]{babel}
\usepackage{microtype}
\usepackage{pdfsync}

\usepackage{amsmath,amssymb,amsfonts,amsthm}
\usepackage{mathtools,accents}
\usepackage{mathrsfs}
\usepackage{aliascnt}
\usepackage{braket}
\usepackage{bm}

\usepackage[a4paper,margin=3cm]{geometry}
\usepackage[citecolor=blue,colorlinks]{hyperref}

\usepackage{enumerate}
\usepackage{xcolor}
\usepackage{xspace}

\makeatletter
\g@addto@macro\@floatboxreset\centering
\makeatother

\makeatletter
\def\newaliasedtheorem#1[#2]#3{
	\newaliascnt{#1@alt}{#2}
	\newtheorem{#1}[#1@alt]{#3}
	\expandafter\newcommand\csname #1@altname\endcsname{#3}
}
\makeatother

\numberwithin{equation}{section}

\newtheoremstyle{slanted}{\topsep}{\topsep}{\slshape}{}{\bfseries}{.}{.5em}{}

\theoremstyle{plain}
\newtheorem{theorem}{Theorem}[section]
\newaliasedtheorem{proposition}[theorem]{Proposition}
\newaliasedtheorem{lemma}[theorem]{Lemma}
\newaliasedtheorem{corollary}[theorem]{Corollary}
\newaliasedtheorem{counterexample}[theorem]{Counterexample}

\theoremstyle{definition}
\newaliasedtheorem{definition}[theorem]{Definition}
\newaliasedtheorem{question}[theorem]{Question}
\newaliasedtheorem{openquestion}[theorem]{Open Question}
\newaliasedtheorem{conjecture}[theorem]{Conjecture}

\theoremstyle{remark}
\newaliasedtheorem{remark}[theorem]{Remark}
\newaliasedtheorem{example}[theorem]{Example}

\newcommand{\setN}{\mathbb{N}}

\newcommand{\setR}{\mathbb{R}}

\newcommand{\eps}{\varepsilon}

\let\altphi\phi
\let\phi\varphi
\let\varphi\altphi
\let\altphi\undefined

\newcommand{\abs}[1]{\left\lvert#1\right\rvert}
\newcommand{\norm}[1]{\left\lVert#1\right\rVert}

\let\div\undefined
\DeclareMathOperator{\div}{div}

\newcommand{\di}{\mathop{}\!\mathrm{d}}

\DeclareMathOperator{\supp}{supp}

\newcommand{\leb}{\mathscr{L}}

\newcommand{\dist}{\mathsf{d}}

\newfont{\tmpf}{cmsy10 scaled 2500}

\newcommand{\T}{{\mathbb{T}}}

\def\Xint#1{\mathchoice
	{\XXint\displaystyle\textstyle{#1}}%
	{\XXint\textstyle\scriptstyle{#1}}%
	{\XXint\scriptstyle\scriptscriptstyle{#1}}%
	{\XXint\scriptscriptstyle\scriptscriptstyle{#1}}%
	\!\int}
\def\XXint#1#2#3{{\setbox0=\hbox{$#1{#2#3}{\int}$ }
		\vcenter{\hbox{$#2#3$ }}\kern-.6\wd0}}

\def\dashint{\Xint-}

\begin{document}
	
	\title{Sobolev estimates for solutions of the transport equation and ODE flows associated to non-Lipschitz drifts}

\author{Elia Bru\'e  \thanks{Scuola Normale Superiore, \url{elia.brue@sns.it, }}~~~~~ 
	Quoc-Hung Nguyen \thanks{New York University Abu Dhabi, \url{qn2@nyu.edu, }} 
	 }
\maketitle

\begin{abstract}
	It is known, after \cite{Jabin16} and \cite{AlbertiCrippaMazzuccato18}, that ODE flows and solutions of the transport equation associated to Sobolev vector fields do not propagate Sobolev regularity, even of fractional order.
	In this paper, we show that some propagation of Sobolev regularity happens as soon as the gradient of the drift is exponentially integrable.
	We provide sharp Sobolev estimates and new examples. As an application of our main theorem, we generalize a regularity result for the 2D Euler equation obtained by Bahouri and Chemin in \cite{BahouriChemin94}.
	
%
	
	\smallskip
	
\end{abstract}
\tableofcontents
\section*{Introduction}
We consider the Cauchy problem for the transport equation associated to a vector field $b:[0,T]\times \T^d\to \setR^d$ on the flat torus $\T^d:=\setR^d/\mathbb{Z}^d$
\begin{equation}\label{CE}
\begin{dcases}
\partial_t u+b\cdot \nabla_x u=0,\\
u(0,x)=u_0(x),
\end{dcases}\tag{Tr}
\qquad 
\end{equation}
where $u_0:\T^d\to \setR$ is a given initial data and $u:[0,T]\times \T^d\to \setR$ is the unknown to the problem.

The theory of characteristics establishes a link between solutions of \eqref{CE} and the flow $X:[0,T]\times \T^d\to \T^d$ of $b$, i.e. the solution of
\begin{equation}\label{ODE}
\begin{dcases}
\frac{\di}{\di t} X(t,x)=b(t,X(t,x))\qquad x\in \T^d, t\in [0,T],\\
X(0,x)=x.
\end{dcases}\tag{ODE}
\end{equation}
Thanks to the classical Cauchy-Lipschitz theory both problems are well-posed when the drift $b$ is regular enough, i.e. Lipschitz in the spatial variable uniformly in time. 
Unfortunately the Lipschitz regularity is a too strong assumption for applications. Indeed in various physical models of the mechanics of fluids it is essential to deal with non regular velocity, and this is not just a technical fact but corresponds to effective physical situations. 
For this reason in the last thirty years a big interest has grown on the study of \eqref{ODE} and \eqref{CE} under weaker assumptions on the vector field.



\paragraph*{}
In the present paper we study sharp regularity properties, in the scale of Sobolev spaces, of solutions of \eqref{CE} and \eqref{ODE} in a setting that is in between the classical setting of the Cauchy-Lipschitz theory and the Sobolev setting considered in the DiPerna-Lions-Ambrosio theory \cite{DiPernalions,Ambrosio04}. More precisely, we assume that $b$ admits a spatial distributional derivative satisfying
\begin{equation}\label{Hpp}
\sup_{t\in [0,T]}	\int_{\T^d}\exp\set{\beta |\nabla_x b(t,x)|} \di x<\infty
\quad\text{for some $\beta>0$ and}\ 
\div_x b\in L^{\infty}([0,T]\times \T^d).\tag{HP}
\end{equation}
We have chosen the ambient space $\T^d$ instead of $\setR^d$ just because compactness allows to avoid integrability problems at infinity and to obtain global estimates. This makes statements shorter and more elegant. It is worth stressing, however, that any result we are going to present holds true also in the Euclidean space $\setR^d$ provided one suitably localizes the estimates.  

The study of \eqref{CE} and \eqref{ODE} under \eqref{Hpp} is meaningful for applications to nonlinear partial differential equations. The 2D Euler equation in vorticity form (see \cite{BertozziMajda02,Lions96} for an overview) provides an important example of PDE where a vector field satisfying \eqref{Hpp} is involved. In particular, as an application of the main result in this work ( \autoref{th: MMain} and \autoref{th: Main 1}) we obtain a propagation of regularity result (\autoref{th:main2}) for solutions of the Euler equation with bounded initial vorticity enjoying a fractional order regularity. This theorem is a non trivial improvement of \cite[Corollary 1.1]{BahouriChemin94} stated in the periodic setting. 
See \autoref{section:Euler} for details on this.





\paragraph*{}
Let us now present the main regularity result of this manuscript underlying, by mean of examples, its sharpness in the Sobolev scale. We refer to \autoref{section:Regularity} for more details on our main theorems \autoref{th: MMain}, \autoref{prop:Regularity of flow} and related corollaries, while the examples \autoref{th: example1}, \autoref{th: example 2} and \autoref{th: example3} are presented in \autoref{section:Examples}.

First of all it is worth mentioning that, under the assumptions \eqref{Hpp}, it is well-known that \eqref{ODE} admits a unique flow in the classical sense.
Indeed the velocity field satisfies the \textit{log-Lipschitz} property that, identifying the drift with a periodic function from $\setR^d$ to $\setR^d$, reads
\begin{equation}\label{eq: log-lipschitz estimate1}
|b(t,x)-b(t,y)|\le C|x-y|\max\left\{|\log\left(|x-y|\right)|,1\right\}
\qquad
\forall x,y\in \setR^d\ t\in [0,T].
\end{equation}
This property implies in turn the existence and uniqueness of the curve $t\mapsto X_t(x)$  satisfying \eqref{ODE} (look at \autoref{lemma: log-lipschitz regularity} and the discussion in \autoref{subsection: regularity of flows}). 
Moreover, $X_t:\T^d\to \T^d$ is invertible for any fixed time $t\in [0,T]$ and
\begin{equation}\label{lagrangianity}
	u(t,x):=u_0((X_t)^{-1}(x))
	\quad t\in [0,T],\quad \text{with}\  u_0\in L^p(\T^d),
\end{equation}
provides the unique distributional solution of the Cauchy problem \eqref{CE} in $L^{\infty}([0,T]; L^p(\T^d))$, for $p\in [1,\infty]$ (see \autoref{remark:existence and uniqueness} for more explanations). For distributional solutions we mean weakly continuous and bounded curves $t\mapsto u_t\in L^p(\T^d)$ satisfying
\begin{equation*}
\int_{\T^d} u_t(x)\phi(x)\di x - \int_{\T^d} u_0(x)\phi(x)\di x=\int_0^t\int_{\T^d} u_s(x)(b_s(x)\cdot \nabla \phi(x)-\div b_s(x)\phi(x))\di x\di s,
\end{equation*}
for any $t\in [0,T]$ and $\phi\in C^{\infty}(\T^d)$.

In order to make this introduction as clear as possible we do not illustrate here our main result \autoref{th: MMain} for \eqref{CE}, since it needs the introduction of a suitable functional class, we refer to \autoref{section:Regularity} for this. We prefer instead focusing the attention on the \textit{Lagrangian} side of the problem (i.e. the study of \eqref{ODE}) that is really the core of our analysis. Indeed any regularity estimate for the flow $X_t$ gives in turn results for the transport equation \eqref{CE} as a consequence of the \textit{Lagrangian identity} \eqref{lagrangianity}. In what follows $\dist$ denotes the intrinsic distance on the flat torus $\T^d$.
\begin{theorem}\label{main in introduction}
	Let $b:[0,T]\times\T^d\to\setR^d$ satisfy
	\begin{equation*}
	\sup_{t\in [0,T]}	\int_{\T^d}\exp\set{\beta |\nabla b(t,x)|} \di x=K<\infty
	\quad\text{for some $\beta>0$ and}\ 
	\norm{\div b}_{L^{\infty}([0,T]\times \T^d)}=L<\infty.
	\end{equation*}
    Then, for any $x,y\in \T^d$ and $t\in [0,T]$, we have
    \begin{equation}\label{i1}
    \dist(X_t(x),X_t(y))\le g_t(x)\dist(x,y)
    \quad
    \text{and}
    \quad
    \dist((X_t)^{-1}(x),(X_t)^{-1}(y))\le g_t(x)\dist(x,y),
    \end{equation}
	for some nonnegative function $g_t$ that fulfills 
	\begin{equation}\label{eq: key1}
	\norm{g_t}_{L^{q_t}}\le e^{\frac{t^2LC_1}{\beta}}( C_2K)^{\frac{tC_1}{\beta}}
	\qquad \forall t\in [0,T],
	\quad \text{with}\quad q_t:=\frac{\beta}{C_1t},
	\end{equation}
	where $C_1>0$ and $C_2>0$ depend only on $d$.
\end{theorem}
\autoref{main in introduction} has to be understood as a quantitative approximation result in the spirit of Lusin's theorem for Sobolev functions (see \cite{Liu77}). Indeed, \eqref{i1} implies that, for any $\lambda>0$, the flow map $X_t$ and its inverse are $\lambda$-Lipschitz if restricted to the set $\set{g_t<\lambda}$. Moreover, since $g_t\in L^{q_t}(\T^d)$, by means of the Chebyschev inequality we can estimate the Lebesgue measure of the ``bad'' set
\begin{equation*}
	\leb^d(\set{g_t\ge\lambda})\le \frac{\norm{g_t}_{L^{q_t}}^{q_t}}{\lambda^{q_t}}
	\le \frac{C_2Ke^{tL}}{\lambda^{q_t}},
\end{equation*}
where we do not control the oscillation of $X_t$.


It is well-known since the work \cite{Hajlasz} that quantitative approximation properties à la Lusin are related (and actually characterize) Sobolev spaces for suitable choices of the exponents. This allows to deduce from \autoref{main in introduction} that
\begin{equation}\label{i2}
	X_t \in W^{1,q_t}(\T^d;\setR^d)
	\qquad\text{for any}\ 0\le t<\frac{\beta}{C_1},
	\quad \text{where}\quad q_t:=\frac{\beta}{C_1t},
\end{equation}
together with the quantitative bound
\begin{equation}\label{i7}
	\norm{\nabla X_t}_{L^{q_t}}\le C_d\norm{g}_{L^{q_t}}\le C_d e^{\frac{t^2LC_1}{\beta}}( C_2K)^{\frac{tC_1}{\beta}}.
\end{equation}
In other words $X_t$ enjoys a definite Sobolev regularity until a critical time that depends only on $\beta$. 
The very same conclusion holds also for $(X_t)^{-1}$ (note that \autoref{main in introduction} gives a symmetric result in $X_t$ and $(X_t)^{-1}$) but, for sake of simplicity, here and in the rest of the introduction we consider just the flow map $X_t$. 

What at the first instance could sound surprising is that \eqref{i2} is \textit{sharp}: it can really happen that the flow associated to a vector field satisfying \eqref{Hpp} ceases to be $W^{1,1}$ regular after a time of order $\sim \beta$.
In \autoref{th: example 2} we build a vector field with such a property. However, instead of explain this example, that is presented in detail in \autoref{section:Examples}, we want to present a formal computation to convey the idea that, if we are in a situation in which the Sobolev regularity of $X_t$ is neither instantaneously lost (as in the DiPerna-Lions setting \cite{AlbertiCrippaMazzuccato18}) nor fully preserved (as in the Cauchy-Lipschitz case), then it reasonably decreases according to \eqref{i2} and \eqref{i7} for structural reasons.

Let us consider a drift $b$ that does not depend on time, so its flow satisfies the semigroup property $X_{t+h}=X_t\circ X_h$. 
If $X_{\delta}\in W^{1,p}$ for some small time $\delta$ and some exponent $1<p<\infty$ then the H\"older inequality suggests that, reasonably, $\nabla X_{n\delta}\in L^{p/n}$ for any integer $n\le p$. Indeed we can use the semigroup property and the chain rule to write
\begin{equation*}
	|\nabla X_{n\delta}|(x)\le |\nabla X_{\delta}|(X_{\delta(n-1)})\cdot |\nabla X_{\delta}|(X_{\delta(n-2)})\cdot...\cdot |\nabla X_{\delta}|(X_{\delta})\cdot|\nabla X_{\delta}|(x),
\end{equation*}
and observe that the right hand side is a product of $n$ functions belonging to $L^p$. More precisely we have 
\begin{equation*}
     \norm{\nabla X_{\delta}(X_{k\delta})}_{L^p}\le e^{k\delta L}\norm{\nabla X_{\delta}}_{L^p}
     \qquad \text{for}\ k=0,...,n-1,
\end{equation*}
where $L:=\norm{\div b}_{L^{\infty}}$. This immediately leads to $\norm{\nabla X_{n\delta}}_{L^{p/n}}\le e^{L\delta n^2}\norm{\nabla X_{\delta}}_{L^p}^n$.
Eventually we set $t:=n\delta$ and rewrite
\begin{equation}\label{i8}
	\norm{\nabla X_{t}}_{L^{\frac{p\delta^{-1}}{t}}}\le e^{L\delta^{-1}t^2}\norm{\nabla X_{\delta}}_{L^p}^{\delta^{-1}t},
	\qquad \text{for}\  0<t\le \frac{p\delta^{-1}}{t}.
\end{equation}
Note that \eqref{i8} is perfectly coherent with \eqref{i2} and \eqref{i7}.


\smallskip

\autoref{main in introduction} allows also to describe the Sobolev regularity of $X_t$ after the critical time $\frac{\beta}{C_1}$.
In this case we can measure the regularity in the scale of \textit{fractional} Sobolev spaces: what happens, roughly, is that $X_t$ admits a derivative of order $\frac{\beta}{Ct}\wedge 1$ in $L^1$ for any $t\in [0,T]$ and again the conclusion is sharp in the scale of Sobolev spaces. Look at \autoref{th: Main 1} for the rigorous statement written in terms of solution of the transport equation and to \autoref{th: example 2} for the example that underlines its sharpness.

\smallskip 


Another simple outcome of \autoref{main in introduction} is the following: if the gradient of the drift satisfies an integrability condition slightly stronger than \eqref{Hpp}, for instance 
\begin{equation*}
\sup_{t\in [0,T]}	\int_{\T^d}\exp\set{\beta |\nabla b_t(x)|} \di x<\infty
\qquad \text{for}\ any\ \beta>0,
\end{equation*}
then $X_t$ belongs to $W^{1,p}$ for any $1\le p<\infty$. Basically it follows from the explicit expressions of $q_t$ and the critical time in \eqref{i2}, look at \autoref{cor:regularity of flows} for more details. 
On the other hand, we have an example (see \autoref{th: example1}) ensuring the existence of a drift satisfying a relaxed version of \eqref{Hpp}, i.e.
\begin{equation*}
\sup_{t>0} \int_{B_2}\exp\left\lbrace \frac{|\nabla b(t,x)|}{\log(1+|\nabla b_t(x)|)^a}\right\rbrace
\di x<\infty
\quad
\forall a>0,
\end{equation*}
whose flow does not belong to any Sobolev space, even of fractional order, for any $t>0$. 
Roughly, it amounts to say that the exponential integrability condition for $\nabla b_t$, that we assume in \eqref{Hpp}, is a threshold condition in order to hope for a Sobolev regularity of the flow map.


The examples we have been mentioning in this introduction are the content of \autoref{section:Examples}; they are all based on a technique introduced recently in \cite{AlbertiCrippaMazzuccato18} by Alberti, Crippa and Mazzucato.

\paragraph*{}

Let us finally spend a few words on the main idea behind the proof of \autoref{main in introduction}. Our strategy builds upon the technique introduced by Crippa and De Lellis in \cite{CrippaDeLellis08} for the quantitative study of generalized flows in the DiPerna-Lions-Ambrosio theory. 
The authors of the present paper have already used similar ideas in \cite{BrueNguyen18A} to obtain sharp regularity estimates for solutions of the continuity equation in the scale of \textit{log-Sobolev} spaces assuming a Sobolev regularity on the drift.
In order to explain a main technical point of the strategy let us recall the standard argument to prove that flow maps inherit the Lipschitz regularity of velocity fields.
When $X_t$ is associated to a uniformly $K$-Lipschitz vector field $b$, using the very definition of flow map, we have
\begin{equation*}
	\frac{\di}{\di t} |X_t(x)-X_t(y)|\le |b_t(X_t(x))-b_t(X_t(y))|\le K |X_t(x)-X_t(y)|
	\qquad \forall x,y\in \setR^d,
\end{equation*}
that together with a Gr\"onwall lemma gives
\begin{equation*}
	|X_t(x)-X_t(y)|\le |x-y|e^{t K }\qquad\forall x, y\in \setR^d,\ t\in [0,T].
\end{equation*}
Note that we have identified both $b$ and $X$ with periodic functions in $\setR^d$.

In order to make a variant of this strategy work in our context we need to consider a weak version of the Lipschitz inequality
\begin{equation}\label{i4}
	|b_t(x)-b_t(y)|\le K|x-y|\qquad\forall x,y\in \setR^d,\ t\in [0,T],
\end{equation}
that is not anymore available assuming just \eqref{Hpp}.

In our setting a natural replacement of \eqref{i4} is the log-Lipschitz property \eqref{eq: log-lipschitz estimate1} that, if plugged in the Gr\"onwall argument above, gives
\begin{equation}\label{i3}
 |X_t(x)-X_t(y)|\le C |x-y|^{e^{-Ct/\beta}},
\end{equation}
see \autoref{subsection: regularity of flows} and the discussion therein for more details.
Even though \eqref{i3} is sharp in the scale of H\"older spaces (see \cite{BahouriChemin94}) it is not suitable for our purposes, indeed it cannot give either integer Sobolev regularity or approximation results by means Lipschitz functions. Moreover \eqref{i3} cannot even implies our result in the case of fractional Sobolev spaces \autoref{th: Main 1} since in our case the regularity dissipates in time with rate $\sim \frac{\beta}{t}$ (that is the sharp rate) while in \eqref{i3} the rate is $\sim e^{-Ct/\beta}$. 
Let us point out that the use of the log-Lipschitz property \eqref{eq: log-lipschitz estimate1} for the study of \eqref{ODE}, \eqref{CE} and related problems coming from PDE nowadays is consider standard, see for instance \cite{BahouriChemin94,CheminLerner95,Zuazua02}.

In this paper we adopt a change of prospective. We forget about the log-Lipschitz property and we take into account a different ingredient that has been already used by Crippa and De Lellis in the Sobolev setting. They have replaced \eqref{i4} with the well-known inequality
\begin{equation}\label{i6}
|b_t(x)-b_t(y)|\le C_d|x-y|(M|\nabla b_t|(x)+M|\nabla b_t|(y)),
\qquad \forall x,y\in \setR^d, \ t\in [0,T],
\end{equation}
available for any Sobolev map (see \cite{Stein} for its proof), where $M$ denotes the Hardy-Littlewood maximal operator. Assuming $\nabla b_t\in L^p$ for $p>1$ one has in turn $M|\nabla b_t|\in L^p$ (it is a general property of the maximal function, see \cite[Theorem 1]{Stein2}) and it leads to a quantitative weak version of \eqref{i4} that is suitable for the study of the regularity of $X_t$.

Under the assumption \eqref{Hpp} we can write a version of \eqref{i6} as follows: there exists a nonnegative function $h_t$ such that 
\begin{equation}\label{i5}
	|b_t(x)-b_t(y)|\le |x-y|\frac{C_d}{\beta}h_t(x)
	\quad \forall x,y\in \T^d,\ t\ge 0
	\quad\text{and}\ \sup_{t\in [0,T]}\int_{\T^d} \exp\left\lbrace h_t(x)\right\rbrace\di x< \infty
\end{equation}
where $C_d$ depends only on $d$, see \autoref{lemma:exponential lusin}.
This technical ingredient is the correct one to replace \eqref{i4} in the Gr\"onwall argument.
We refer to \autoref{section:Regularity} for more details.


\paragraph{Notations.}
We denote by $\T^d:=\setR^d/\mathbb{Z}^d$ the flat torus of dimension $d\ge 1$ endowed with its geodesic distance $\dist$ and its Haar measure $\leb^d$. We denote by $B_r(x)$ the geodesic ball of radius $r>0$ centered at $x\in\T^d$.

We often identify $\T^d$ with $[0,1)^d$, in this way we can write
\begin{equation}\label{eq:distance}
\dist(x,y):=\min\set{|x-y-k|\ :\ k\in \mathbb{Z}^d\ |k|\le 2},
\end{equation} 
where $|\ \cdot\ |$ is the Euclidean distance in $\setR^d$. Under this identification the Haar measure in $\T^d$ coincides with the Lebesgue measure on the square, while  scalar functions can be identified $f:\T^d\to \setR$ with $1$-periodic functions on $\setR^d$. 
We often use the double notation $f(t,x)=f_t(x)$ for functions depending both in the space and in the time variable.

We write
\begin{equation*}
\dashint_E f\di \mu =\frac{1}{\mu(E)} \int_E f \di x,
\end{equation*}
to denoted the average integral and
\begin{equation*}
Mf(x):=\sup_{r>0} \dashint_{B_r(x)} |f(y)| \di y,
\qquad
\forall~x\in \mathbb{R}^d,
\end{equation*}
to denote the Hardy-Littlewood maximal function.

We often use the expression $a\lesssim_c b$ to mean that there exists a universal constant $C$ depending only on $c$ such that $a\leq C b$. The same convention is adopted for $\gtrsim_c$ and $\simeq_c$.


\section{Regularity results}\label{section:Regularity}
In this section we present regularity results for flows and solutions of the transport equation associated to drifts satisfying \eqref{Hpp}. Let us begin by introducing a functional class.

\begin{definition}\label{def: F}
	Let $0<\alpha\le 1$ and $0<p\le \infty$ be fixed.
	We say that $f\in L^p(\T^d)$ belongs to $F^{\alpha}_p$ if
	\begin{equation*}
	[f]_{F^{\alpha}_p}:=
	\inf \left\lbrace \norm{g}\in L^p(\T^d):\ |f(x)-f(y)|\le\dist(x,y)^{\alpha}(g(x)+g(y))
	\quad \text{for every}\ x,y\in \T^d
	\right\rbrace<\infty.
	\end{equation*} 
	We set $\norm{f}_{F^{\alpha}_p}:=\norm{f}_{L^p}+[f]_{F^{\alpha}_p}$. 
\end{definition}
These spaces have already appeared in the literature (see for instance \cite{BahouriChemin94}) and they coincide with the Triebel-Lizorkin class $F^{\alpha}_{p,\infty}$ when $\alpha\in (0,1)$ and $p>1$ (see \cite[Proposition 3.2]{BahouriChemin94}).
The Hajlasz characterization of Sobolev spaces \cite{Hajlasz} gives 
\begin{equation}\label{F vs Sobolev}
F_p^1= W^{1,p}(\T^d)\quad\text{for any}\ p>1,
\quad\text{and}\quad 
F^1_1\subset W^{1,1}(\T^d).
\end{equation}
While, for $0<\alpha<1$, the class $F^{\alpha}_p$ is related to fractional Sobolev spaces (see \cite{AdamsFournier75})
\begin{equation*}
	 W^{\alpha,p}(\T^d):=\set{f\in L^p(\T^d):\ [f]_{W^{\alpha,p}}<\infty}
	 \qquad
	 \alpha\in (0,1),\ p\ge 1
\end{equation*}
where
\begin{equation}
[f]_{W^{\alpha, p}}:=\left( \int_{(0,1]^d} \int_{\T^d} \frac{|f(x+h)-f(x)|^p}{|h|^{d+ps}} \di x \di h \right)^{1/p},
\end{equation}
is the socalled Gagliardo's seminorm. Precisely we have
\begin{equation}\label{F vs W}
	W^{\alpha,p}(\T^d)\subset F^{\alpha}_p \subset W^{\alpha^\prime,p}(\T^d)
	\qquad
	\text{for any}\ 0<\alpha^\prime<\alpha<1,\ p>1,
\end{equation}
the proof of the first inclusion follows form \cite[Proposition 1.13]{BrueNguyen18B} while the latter can be easily checked using the definition of $F_p^{\alpha}$ and Gagliardo's seminorm.

Let us finally mention that, the inequality
\begin{equation*}
	|f(x)-f(y)|=|f(x)-f(y)|^{\theta}|f(x)-f(y)|^{1-\theta}\lesssim \dist(x,y)^{\alpha\theta}(g(x)^{\theta}+g(y)^{\theta})\norm{f}_{L^{\infty}}^{1-\theta},
\end{equation*}
for any $\theta\in (0,1)$, $f\in F^{\alpha}_p$ and $g$ competitors in the definition on $[f]_{F^{\alpha}_p}$ (see \autoref{def: F}), implies the interpolation estimate
\begin{equation}\label{interpolation}
	[f]_{F^{\theta\alpha}_p}\lesssim \norm{f}_{L^{\infty}}^{1-\theta}[f]_{F_{p\theta}^{\alpha}}^{\theta}
	\qquad
	\text{for any}\ \theta\in (0,1), \alpha\in (0,1), \ p>0,
\end{equation}
that will play a role in the sequel.

This being said we are ready to state our main result.
\begin{theorem}\label{th: MMain} 
	Let $b:[0,T]\times\T^d\to\setR^d$ satisfy
     \begin{equation*}
     \sup_{t\in [0,T]}	\int_{\T^d}\exp\set{\beta |\nabla b_t(x)|} \di x=K<\infty
     \quad \text{for some $\beta>0$ and}\  
     \norm{\div b}_{L^{\infty}([0,T]\times \T^d)}=L<\infty.
     \end{equation*}
     where the derivatives are understood in the sense of distributions. Then, there exist constants $C_1>0$ and $C_2>2$ depending only on $d$, such that
	for any $0<\alpha \le 1$, $p\ge 1$ and $u_0\in F^{\alpha}_p$ the unique solution $u\in L^{\infty}([0,T];L^p( \T^d))$ of \eqref{CE} satisfies
	\begin{equation}\label{Z1}
	[u_t]_{F^{\alpha}_{p_t}} \le [u_0]_{F^{\alpha}_p}(e^{Lt}C_2K)^{1/p_t}
	\quad
	\forall t\in[0,T]
	\quad
	\text{with}
	\quad
	p_t:=\frac{p}{1+\beta^{-1}\alpha p C_1 t}.
	\end{equation}
\end{theorem}

Before proving \autoref{th: MMain} we present a remark and two important corollaries.

\begin{remark}\label{remark:existence and uniqueness}
	Let us explain why under the assumption \eqref{Hpp} the Cauchy problem \eqref{CE} admits
	\begin{equation}
	u(t,x):=u_0((X_t)^{-1}(x))
	\quad t\in [0,T],
	\end{equation}
	where $X$ is the flow map of $b$ (see the discussion in \autoref{subsection: regularity of flows} for what concerns $X$), as a unique solution in $L^{\infty}([0,T];L^p(\T^d))$, for any $p\ge 1$.
	 
	First of all notice that $u(t,x):=u_0((X_t)^{-1}(x))$ is a weak solution of the transport equation in $L^{\infty}([0,T];L^p(\T^d))$ when $u_0\in L^p(\T^d)$ (look at the introduction for the definition of weak solution). Therefore to prove the sought claim it suffices to show the uniqueness property for \eqref{CE} in the class $L^{\infty}([0,T];L^1(\T^d))$.
	
	Using \autoref{lemma: log-lipschitz regularity} we deduce that $b$ is Log-Lipschitz continuous and, if we further assume that $\div b=0$, then \cite[Theorem 1.2]{BahouriChemin94} grants the uniqueness result we are looking for. It actually implies uniqueness in the larger class of signed measure, but we are not interested in this general case.
	In order to get rid of the assumption $\div b=0$ we can consider the recent result \cite[Theorem 1.1 and Remark 1.5]{CaCr18} together with the simple observation that in our case forward-backward curves are always trivial due to the pointwise uniqueness of trajectories in \eqref{ODE}.
\end{remark}

\begin{corollary}\label{th: Main 1} 
Let $b$, $L$, $C_1$ and $C_2$ be as in \autoref{th: MMain}.	
Then for any $0<\alpha \le 1$, $p\ge 1$ and $u_0\in L^{\infty}(\T^d)\cap W^{\alpha,p}(\T^d)$ the unique solution $u\in L^{\infty}([0,T]\times \T^d)$ of \eqref{CE} satisfies
\begin{equation}\label{conclusion 1}
	u_t\in W^{\alpha_t,p}(\T^d)
	\qquad
	\text{with}
	\quad
	\alpha_t:=\alpha\  \frac{1}{1+2\beta^{-1}\alpha p C_1 t}.
\end{equation}
Moreover, if $p>1$, for any $1\le p^\prime<p$ and $0<\alpha^\prime<\alpha$ it holds
\begin{equation}\label{conclusion 2}
	u_t\in W^{\alpha^\prime,p^\prime}(\T^d)
	\qquad
	 t< \left(\frac{1}{p^\prime}-\frac{1}{p}\right)\frac{\beta}{\alpha C_1}.
\end{equation}
Finally, if we assume $\alpha=1$ the conclusion \eqref{conclusion 2} can be strengthen as follows
\begin{equation}\label{conclusion 3}
\norm{\nabla u_t}_{L^{p'}}\lesssim_{d,p} \norm{\nabla u_0}_{L^p} (e^{Lt} C_2 K)^{\frac{1+\beta^{-1}pC_1 t}{p}}
\qquad
 t< \left(\frac{1}{p'}-\frac{1}{p}\right)\frac{\beta}{C_1},
\end{equation}
for any $1\le p'<p$.
\end{corollary}
\begin{proof}
	Using \eqref{interpolation} with $\theta_t:=\frac{1}{1+\beta^{-1}\alpha p C_1 t}$ and \autoref{th: MMain} we get
	\begin{equation*}
		[u_t]_{F^{\theta_t \alpha}_p}\lesssim \norm{u_0}_\infty^{1-\theta_t}[f]_{F^{\alpha}_{p_t}}^{\theta_t}
		\le \norm{u_0}_\infty^{1-\theta_t}[u_0]_{F^{\alpha}_p}^{\theta_t}(e^{Lt}C_2K)^{1/p},
	\end{equation*}
	since $\theta_t\alpha>\alpha_t$ \eqref{conclusion 1} follows from \eqref{F vs W}.
	Let us address \eqref{conclusion 2}. If $t< \left(\frac{1}{p'}-\frac{1}{p}\right)\frac{\beta}{\alpha C_1}$ then $p_t>p'$, thus \autoref{th: MMain} and \eqref{F vs W} gives
	\begin{equation*}
	[u_t]_{W^{\alpha',p'}}\overset{\eqref{F vs W}}\lesssim_{\alpha,\alpha',p,p'}	[u_t]_{F^{\alpha}_{p'}}\lesssim [u_t]_{F^{\alpha}_{p_t}}\overset{\eqref{Z1}}\lesssim [u_0]_{F^{\alpha}_p}(e^{Lt}C_2K)^{1/p_t}
	\qquad
	\forall\ 0<\alpha'<\alpha.
	\end{equation*} 
	Repeating the same argument with $\alpha=1$ and taking into account \eqref{F vs Sobolev} we get \eqref{conclusion 3}.
\end{proof}

\begin{corollary}\label{cor: main1}
	Let $b:[0,T]\times\T^d\to\setR^d$ satisfy
    \begin{equation*}
     \sup_{t\in [0,T]}	\int_{\T^d}\exp\set{\beta |\nabla b_t(x)|} \di x<\infty
     \quad \text{for \textit{every} $\beta>0$ and}\  
     \norm{\div b}_{L^{\infty}([0,T]\times \T^d)}<\infty.
    \end{equation*}
    where derivatives are understood in the sense of distributions. Then 
	for any $0<\alpha \le 1$, $p> 1$ and $u_0\in L^{\infty}(\T^d)\cap W^{s,p}(\T^d)$ the unique solution $u\in L^{\infty}([0,T]\times \T^d)$ of \eqref{CE} satisfies
    \begin{equation}\label{cor conclusion1}
    	u_t\in W^{\alpha',p}(\T^d)
    	\qquad
   	\text{for any}\quad 0<\alpha'<\alpha,\quad   t\in [0,T].
    \end{equation}
	In the case $\alpha=1$ we also have
     \begin{equation}\label{cor conclusion 2}
     	u_t\in W^{1,p'}(\T^d)
     	\qquad
     \text{for any}\quad 1\le p'<p,\quad  t\in [0,T].
     \end{equation}
\end{corollary}

\begin{proof}
	Let us first assume $p>1$.
	An immediate application of \autoref{th: MMain} gives $u_t\in F^{\alpha}_{p'}$ for any $1\le p'<\infty$ thus the sought conclusions follow from \eqref{F vs Sobolev} and \eqref{F vs W}. In order to extend \eqref{cor conclusion1} to the case $p=1$ it is enough to apply the Sobolev embedding theorem \eqref{eq:Sobolev embedding} stated below.
\end{proof}

\begin{remark}\label{remark:improvement p=1}
	The conclusions \eqref{conclusion 2} can be extended to the case $p=1$
	as follows: for any $\alpha'<\alpha$ there exists $\bar T:=\bar T(\alpha',\alpha, p, \beta, C_1)>0$ such that
	\begin{equation*}
	u_t\in W^{\alpha',p}
	\qquad
	\forall t<\bar T.
	\end{equation*}
	In order to do so it is enough to use the Sobolev embedding theorem:
		\begin{equation}\label{eq:Sobolev embedding}
		\norm{f}_{W^{s,r}}\lesssim_{s,s',r,r'} \norm{f}_{W^{s',r'}},
         ~~
		\text{with}
		~~
		r'=\frac{dr}{d-(s-s')r},\ s,s'\in (0,1],\ r,r'\in (0,\infty),
		\end{equation}
		where $\norm{f }_{W^{\alpha,p}}:=\norm{f}_{L^p}+[f]_{W^{\alpha,p}}$.
		We refer to \cite{AdamsFournier75} and \cite{Daodiaznguyen} for more details.	
\end{remark}

The remaining part of this section is dedicated to the proof of \autoref{th: MMain}. As we have anticipated in the introduction, we carry out a Lagrangian approach, meaning that the core of our argument is a regularity result for flows, whose proof is based on a technique introduced by Crippa and De Lellis \cite{CrippaDeLellis08} in the context of DiPerna-Lions-Ambrosio's theory \cite{DiPernalions, Ambrosio04}.

\subsection{Regularity of flows}\label{subsection: regularity of flows}

In this subsection we prove \autoref{main in introduction}. It is restated below for reader's convenience.

\begin{theorem}[Regularity of the flow]\label{prop:Regularity of flow}
	Let $b:[0,T]\times\T^d\to\setR^d$ satisfy
	\begin{equation*}
	\sup_{t\in [0,T]}	\int_{\T^d}\exp\set{\beta |\nabla b(t,x)|} \di x=K<\infty
	\quad\text{for some $\beta>0$ and}\ 
	\norm{\div b}_{L^{\infty}([0,T]\times \T^d)}=L<\infty.
	\end{equation*}
	Then, for any $x,y\in \T^d$ and $t\in [0,T]$, we have
	\begin{equation}\label{eq:regflow}
	\dist(X_t(x),X_t(y))\le g_t(x)\dist(x,y)
	\quad
	\text{and}
	\quad
	\dist((X_t)^{-1}(x),(X_t)^{-1}(y))\le g_t(x)\dist(x,y),
	\end{equation}
	for some nonnegative function $g_t$ that fulfills 
	\begin{equation}\label{eq: key}
	\norm{g_t}_{L^{q_t}}\le e^{\frac{t^2LC_1}{\beta}}( C_2K)^{\frac{tC_1}{\beta}}
	\qquad \forall t\in [0,T],
	\quad \text{with}\quad q_t:=\frac{\beta}{C_1t},
	\end{equation}
	where $C_1>0$ and $C_2>0$ depend only on $d$.
\end{theorem}
Before proving \autoref{prop:Regularity of flow}, let us recall that, under the assumption
\begin{equation*}
  \sup_{t>0}\int_{\T^d}\exp\set{\beta |\nabla b_t(x)|} \di x=:K(\beta)<\infty
  \qquad
  \text{for some}\ \beta>0,	
\end{equation*}
there exists a unique classical solution of the problem \eqref{ODE}. Indeed thanks to Lemma \eqref{lemma: log-lipschitz regularity} we know that $b$ is Log-Lipschitz, namely 
\begin{equation}\label{z13}
|b_t(x)-b_t(y)|\le \frac{C}{\beta}\dist(x,y)\log\left( \frac{CK}{\dist(x,y)^d} \right)
\qquad
\forall x,y\in \T^d\ \forall t\ge 0.
\end{equation}
In particular $b$ satisfies the Osgood condition, so it admits a unique solution for any initial data $x\in \T^d$.

Moreover, by mean of \eqref{z13}, it is possible to show that $X_t$ is H\"older continuous:
\begin{equation}\label{eq: Holder flow}
\left(CK \right)^{-1-e^{Ct/\beta}} \dist(x,y)^{e^{Ct/\beta}}\le \dist(X_t(x),X_t(y))\le \left(CK \right)^{1+e^{-Ct/\beta}} \dist(x,y)^{e^{-Ct/\beta}},
\end{equation}
for some $C>2$ depending only on $d$. To see this we use again \eqref{z13} obtaining
\begin{equation*}
\abs{\frac{\di}{\di t} \dist(X_t(x),X_t(y))}\le |b_t(X_t(x))-b_t(X_t(y))|
\le \frac{C}{\beta}\dist(X_t(x),X_t(y))\log\left( \frac{CK}{\dist(X_t(x),X_t(y))^d} \right),
\end{equation*}
that amounts to
\begin{equation}\label{zzz2}
\left|\frac{\di}{\di t}\log\log\left( \frac{CK}{\dist(X_t(x),X_t(y))^d}\right)\right|\leq \frac{C}{\beta},
\end{equation}
where the constant in the left hand side may be bigger then the one in the previous line but still depends only on $d$. Thus 
\eqref{zzz2} immediately implies \eqref{eq: Holder flow}.

If we further assume $\norm{\div b}_{L^{\infty}}:=L<\infty$ we also deduce
\begin{equation}\label{eq: compressibility}
e^{-tL}\leb^d \le (X_t)_{\#} \leb ^d\le e^{tL}\leb^d
\qquad
\text{for any}\ t\ge 0.
\end{equation}
In particular when $b$ is divergence-free,  $X_t$ is a measure preserving map for any $t\ge 0$.
The property \eqref{eq: compressibility} can be checked observing that $X$ coincides with the unique \textit{Regular Lagrangian flow} associated to $b$ according to Ambrosio's axiomatization (see \cite{Ambrosio04}).

Let us refer to \cite{BahouriChemin94}, \cite{CheminLerner95} and \cite{Zuazua02} for further details on well-posedness results for flows and solutions of the continuity and transport equation associated to Log-Lipschitz drifts.

We conclude this subsection by proving \autoref{prop:Regularity of flow} and stating a simple corollary.

\begin{proof}[Proof of \autoref{prop:Regularity of flow}]
	Fix $x,y\in \T^d$, recalling \eqref{eq:distance} we have
	\begin{align*}
	\abs{\frac{\di}{\di t} \dist(X_t(x),X_t(y))}\le & |b_t(X_t(x))-b_t(X_t(y))|\\
	\le & \dist(X_t(x),X_t(y))\frac{2C_1}{\beta} \left(1+\log\left(c_dM\left(\exp\left\{2^{-1}\beta |\nabla b_t|\right\}\right)(X_t(x))\right) \right),
	\end{align*}
	where in the second line we used \eqref{lemma:exponential lusin} with $\beta'=\beta/2$. 
	Setting 
	\begin{equation*}
		g_t(x):=\exp\left\lbrace  \frac{2C_1}{\beta}\int_0^t \left(1+\log\left(c_dM\left( \exp\left\{2^{-1}\beta |\nabla b_s|\right\}\right)(X_s(x))\right) \right)\di s
		\right\rbrace,
	\end{equation*}
    the Gr\"onwall inequality gives \eqref{eq:regflow}.
	It remains to prove \eqref{eq: key}. Let us fix $t>0$ and set $q_t:=\frac{\beta}{C_1 t}$, using Jensen's inequality and \eqref{eq: compressibility} we deduce
	\begin{align*}
	\int_{\T^d} g_t(x)^{q_t} \di x = & \int_{\T^d} \exp\left\lbrace 2+\dashint_0^t 2\log\left(c_dM\left(\exp\left\{2^{-1}\beta |\nabla b_s|\right\}\right)(X_s(x))\right)\di s \right\rbrace \di x\\
	 \lesssim_d & \int_{\T^d} \dashint_0^t \left[ M\left(\exp\left\{2^{-1}\beta |\nabla b_s|\right\}\right)(X_s(x))\right]^2\di s \di x\\
	\lesssim_d & e^{tL}\dashint_0^t \int_{\T^d} \left[ M\left(\exp\left\{2^{-1}\beta |\nabla b_s|\right\}\right)(x)\right]^2\di s \di x.
	\end{align*}
	Exploiting the boundness of the maximal function between $L^2$ spaces (see \cite{Stein}) we get the sought conclusion:
    \begin{align*}
	\int_{\T^d} g_t(x)^{q_t} \di x\lesssim_d  e^{tL} \dashint_0^t \int_{\T^d} \exp\left\lbrace \beta |\nabla b_s(x)|\right\rbrace \di x \di s
	\le   e^{tL} K.
	\end{align*}
\end{proof}
An immediate consequence of \autoref{prop:Regularity of flow} is the following.
\begin{corollary}\label{cor:regularity of flows}
	Let $b:[0,T]\times\T^d\to\setR^d$ satisfy
	\begin{equation*}
	\sup_{t\in [0,T]}	\int_{\T^d}\exp\set{\beta |\nabla b_t(x)|} \di x<\infty
	\quad \text{for \textit{every} $\beta>0$ and}\  
	\norm{\div b}_{L^{\infty}([0,T]\times \T^d)}<\infty.
	\end{equation*}
	 Then, for any $t\in [0,T]$, $X_t$ and its inverse belong to $W^{1,p}(\T^d;\setR^d)$ for every $1\le p<\infty$.
\end{corollary}
\begin{proof}
	Let us fix $t>0$. Using \eqref{eq: key} with $\beta=pC_1 t$ we deduce $g_t\in L^p(\T^d)$. The sought conclusion follows from \eqref{F vs Sobolev}.
\end{proof}

We conclude the section with the proof of \autoref{th: MMain}.

\subsection{Proof of \autoref{th: MMain}}
    We assume without loss of generality that $[u_0]_{F^{\alpha}_p}=1$.
	It is enough to prove that, for any $t\in [0,T]$, there exists a positive function $\bar g_t$ such that
	\begin{equation}\label{v1}
		|u_t(x)-u_t(y)|\lesssim \dist(x,y)^{\alpha}(\bar g_t(x)+\bar g_t(y))
		\qquad
		\forall x,y\in \T^d,
	\end{equation}
	and
	\begin{equation}\label{v2}
		\norm{\bar g_t}_{L^{p_t}}\lesssim_{d,p,\alpha} (e^{tL} C_2 K)^{1/p_t}
		\qquad
		\text{with}
		\quad
	    p_t=\frac{p}{1+\beta^{-1}\alpha p C_1 t},
	\end{equation}
	where $C_1$ and $C_2$ are as in \autoref{prop:Regularity of flow}.

	As we mentioned before we exploit the Lagrangian representation formula
	\begin{equation}\label{z15}
		u_t(x):=u_0((X_t)^{-1}(x))
		\qquad
		\text{for any}\ x\in \T^d,\ \text{and}\ t\in[0,T],
	\end{equation}
	where $X_t$ is the solution of \eqref{ODE}. Note that the inverse of $X_t$ is well-defined thanks to \autoref{prop:Regularity of flow}.
	By \autoref{def: F} we know that there exists $h:\T^d\to [0,\infty]$ satisfying
	\begin{equation}\label{z14}
	|u_0(x)-u_0(y)|\le \dist(x,y)^{\alpha}(h(x)+h(y))
	\qquad
	\forall x,y\in \T^d
	\quad\text{and}\quad 
	\norm{h}_{L^p}\le 1+\eps,
	\end{equation}
	where $0<\eps<1$ is fixed.
	Building upon \eqref{z15}, \eqref{z14} and \autoref{prop:Regularity of flow}(ii) we get
	\begin{align*}
		|u_t(x)-u_t(y)| = & |u_0((X_t)^{-1}(x))-u_0((X_t)^{-1}(y))|\\
		                \le & \dist((X_t)^{-1}(x),(X_t)^{-1}(y))^{\alpha}(h((X_t)^{-1}(x))+h((X_t)^{-1}(y)))\\
		                \le & \dist(x,y)^{\alpha} g_t(x)^{\alpha}(h((X_t)^{-1}(x))+h((X_t)^{-1}(y))).
	\end{align*}
	Setting $q_t=\frac{\beta}{C_1 t}$, where $C_1$ is as in \autoref{prop:Regularity of flow}, and using the Young inequality with exponents $\left(\frac{q_t}{\alpha p_t},\frac{p}{p_t}\right)$ we deduce
	\begin{align*}
		g_t(x)^{\alpha}&(h((X_t)^{-1}(x))+h((X_t)^{-1}(y)))\\
		\lesssim & \left(\frac{\alpha p_t}{q_t} g_t(x)^{q_t/p_t}+\frac{p_t}{p}
		            (h((X_t)^{-1}(x))^{p/p_t}\right)
		+\left(\frac{\alpha p_t}{q_t} g_t(y)^{q_t/p_t}+\frac{p_t}{p}
		            (h((X_t)^{-1}(y))^{p/p_t}\right)\\
		=: & \bar g_t(x)+\bar g_t(y).
	\end{align*}
	Thanks to \eqref{eq: compressibility} and \eqref{eq: key} we get
	\begin{align*}
		\norm{\bar g_t}_{L^{p_t}}\le & 
		\frac{\alpha p_t}{q_t}\norm{g_t}_{L^{q_t}}^{q_t/p_t}+\frac{p_t}{p}e^{\frac{tL}{p_t}}\norm{h}_{L^p}^{p/p_t}
		\le  e^{\frac{tL}{p_t}}(1+\eps)^{p/p_t}\left(\frac{\alpha p_t}{q_t}(C_2K)^{1/p_t}+\frac{p_t}{p}\right)\\
		\le & (1+\eps)^{p/p_t}(e^{tL} C_2 K)^{1/p_t},
	\end{align*}
	letting $\eps\to 0$ we conclude the proof.

\section{Counterexamples}\label{section:Examples}
In this section we prove that \autoref{th: Main 1} and \autoref{prop:Regularity of flow} are optimal in the scale of Sobolev spaces by mean of three different examples.
The first one tries to answer the question whether the integrability condition
\begin{equation*}
\sup_{t>0}	\int_{\T^d}\exp\set{\beta |\nabla b_t(x)|} \di x<\infty
\qquad\text{for some}\ \beta>0,
\end{equation*}
assumed in \autoref{th: MMain} and \autoref{prop:Regularity of flow} can be relaxed.
\begin{theorem}\label{th: example1}
	There exist a divergence free velocity field $b$ satisfying
	\begin{equation}\label{ine1}
	    \supp b_t\subset  B_{1/2}\quad  \forall t\ge 0;
	    \quad
	    \quad
		\sup_{t>0} \int_{B_{1/2}}\exp\left\lbrace \frac{|\nabla b_t(x)|}{\log(1+|\nabla b_t(x)|)^a}\right\rbrace
		\di x<\infty
		\quad
	    \forall a>0,
	\end{equation}
	and $u_0\in C_c^{\infty}(B_{1/2})$ such that $\supp u_t\subset  B_{1/2}~\forall t\ge 0$ and 
	\begin{equation*}
		[u_t]_{W^{s,p}}=\infty
		\qquad
		\forall t>0, \ \forall s>0, \ \forall p\ge 1,
	\end{equation*}
	where $u_t$ is the solution in $L^{\infty}([0,\infty)\times \setR^d)$ of \eqref{CE} with initial data $u_0$.
\end{theorem}
The second example shows that the conclusions in \autoref{th: Main 1} are sharp in the scale of fractional Sobolev spaces.

\begin{theorem}\label{th: example 2}
     For any $m\in \setN$ and $\lambda>0$ there exist 
     a divergence free velocity field $b$ satisfying
     \begin{equation}\label{zz1}
      \supp b_t\subset  B_{1/2}\quad  \forall t\ge 0;
     \quad
     \quad
     \sup_{t>0} \int_{B_{1/2}}\exp\left\lbrace\lambda|\nabla b_t(x)| \right\rbrace \di x<\infty
     \end{equation}
     and $u_0\in C_c^m(B_{1/2})$ such that the unique solution $u\in L^{\infty}([0,\infty)\times \setR^d)$ of \eqref{CE} with $\supp u_t\subset  B_{1/2}$ for any $t\ge 0$ fulfills
     \begin{itemize}
     	\item[(i)]  $\norm{u_t}_{W^{1,1}}=\infty$ for any $t>c_1\lambda$;
     	\item[(ii)] $[u_t]_{W^{\alpha_t,1}}=\infty$ for any $t>0$, where $\alpha_t:=1\wedge \frac{c_2 \lambda}{t}$ for any $t>0$;
     \end{itemize}
     where $c_1>0$ and $c_2>0$ depend only on $m$ and $d$.
\end{theorem}
\begin{remark}
	An immediate consequence of \autoref{th: example 2} is that the flow map $X_t$ associated to $b$ satisfies $\norm{\nabla X_t}_{W^{1,1}}=\infty$ for any $t>c\lambda$, where $c=c(d)>0$.
\end{remark}

The last example shows that, in general, we cannot hope for the Lipschitz regularity of the flow associated to $b$ when $\norm{\nabla b}_{L^{\infty}}=\infty$, even under a very strong integrability assumption (in the Orlicz sense)  on $\nabla b$. In particular we cannot extend \autoref{cor:regularity of flows} and \eqref{cor conclusion 2} to the case $p=\infty$.

\begin{theorem}\label{th: example3}
	Let us fix an increasing $\Phi:[0,\infty)\to [0,\infty)$ satisfying $\Phi(0)=0$ and $\lim_{r\to\infty}\Phi(r)=\infty$.
	Then there exist a divergence free velocity field $b$ satisfying
	\begin{equation}\label{zz2}
	\supp b_t\subset  B_{1/2}\quad  \forall t\ge 0;
	\quad
	\quad
	\sup_{t>0} \int_{B_{1/2}} \Phi(|\nabla b_t(x)|) \di x<\infty,
    \end{equation}	
    and  $u_0\in C_c^1(B_{1/2})$ such that $\supp u_t\subset  B_{1/2}~\forall t\ge 0$ and 
    \begin{equation*}
    	\norm{\nabla u_t}_{L^{\infty}}=\infty
    	\qquad
    	\forall t>0,
    \end{equation*}
    where $u_t\in L^{\infty}([0,T)\times \setR^d)$ is the solution of \eqref{CE} with initial data $u_0$.
\end{theorem}

Let us spend a few words explaining the idea behind the construction of the examples in \autoref{th: example1}, \autoref{th: example 2} and \autoref{th: example3}. Basically, they are built following a common strategy that has been introduced for a first time in \cite{AlbertiCrippaMazzuccato14}, \cite{AlbertiCrippaMazzuccato18} and recently adopted in \cite{BrueNguyen18A}.
Following this scheme the construction of the vector field $b$ and the solution $u_t$ of \eqref{CE} is achieved by patching together a countable number of pairs $v_n$ and $\rho_n$ of velocity fields and solutions to \eqref{CE} with disjoint supports. They are obtained by rescaling in space, time and size $v$ and $\rho$, that are the fundamental building block provided by \autoref{alberti,Crippa,Mazzucato}. 
Choosing properly the scaling parameters we get the three different examples.

\begin{proposition}\label{alberti,Crippa,Mazzucato}
	 Assume $d\geq 2$ and let $Q$ be the open cube with unit side centered at the origin of $\mathbb{R}^d$. There exist a velocity field $v\in C^\infty([0,\infty)\times\mathbb{R}^d)$ and a solution $\rho\in  L^\infty([0,\infty)\times\mathbb{R}^d) $ of \eqref{CE} such that 
	\begin{itemize}
		\item[(i)]$v_t$ is bounded, divergence-free and compactly supported in $Q$ for any $t\ge 0$;
		\item[(ii)] $\rho_t$ has zero average and it is bounded and compactly supported in $Q$ for any $t\ge 0$;
		\item[(iii)] $\sup_{t\ge 0}\norm{v_t}_{W^{1,\infty}(\setR^d)}<\infty$ for any $t\ge 0$, for any $1\le p\le \infty$;
		\item[(iv)] there exists a constant $c>0$ such that
		\begin{equation}\label{z11}
		\norm{\rho_t}_{\dot W^{s,p}(\mathbb{R}^d)}\gtrsim \exp(cst),
		\qquad \forall t\ge 0,\quad s>0,\quad 1\leq p\leq \infty.
		\end{equation}
	\end{itemize}	
\end{proposition}
\begin{proof} 
	As remarked in \cite[Remark 10]{AlbertiCrippaMazzuccato18} we can assume $d=2$. In \cite{AlbertiCrippaMazzuccato16} the authors proved the existence of a velocity field $v\in C^\infty([0,\infty)\times\mathbb{R}^d)$ and a solution $\rho\in  L^\infty([0,\infty)\times\mathbb{R}^d) $ of \eqref{CE} satisfying (i), (ii), (iii) and 
	\begin{equation*}
	\norm{\rho_t}_{H^{-s}(\setR^2)}\leq C_s\exp(-sct),
	\qquad \text{for any}\ s\in (0,1).
	\end{equation*}	
     However, from \cite[page 33, proof of 6.4]{AlbertiCrippaMazzuccato16}, we also have 
		\begin{equation*}
	\norm{\rho_t}_{W^{-s,q}(\setR^2)}\lesssim_{s,q}\exp(-cst)
	\quad \text{for any }1\leq q\leq 2,\ \text{and}\ 0<s<2.
	\end{equation*}
	Therefore, thanks to Gagliardo–Nirenberg interpolation inequality (see \cite{AdamsFournier75}) we obtain \eqref{z11}. 	
\end{proof}

Before going into details with the proofs of \autoref{th: example1}, \autoref{th: example 2} and \autoref{th: example3} we present a technical ingredient.
\begin{lemma} \label{subesti}
	Let $\gamma\in (-\infty,1)$ be fixed. For every $n\in \setN$ consider an open set $\Omega_n$, a function $f_n\in L^p(\setR^d)$ and a parameter $0<\lambda_n<1/4$. Assume that the family $\set{\Omega_n}_{n\in \setN}$ is disjoint and that the distance between $\supp f_n$ and $\setR^d\setminus \Omega_n$ is bigger than $\lambda_n$ for every $n\in \setN$.
	Then it holds
	\begin{align}\nonumber
	\int_{\mathbb{R}^d} & \int_{\mathbb{R}^d}\frac{|\sum_n f_n(x+h)-\sum_n f_n(x)|^p}{|h|^{d+sp}}\di x\di h\\
	&\ge \limsup_{N\to\infty} \sum_{n=1}^N \left( \int_{\mathbb{R}^d}\int_{\mathbb{R}^d}\frac{|f_n(x+h)-f_n(x)|^p}{|h|^{d+sp}} \di x\di h-\frac{c(d)2^{p}}{sp} \left(\frac{2}{\lambda_n}\right)^{sp}\norm{f_n}_{L^p}^p\right).
	\label{essumf}
	\end{align}
\end{lemma}

\begin{proof} 
	Let us call $\bar \Omega_n\subset \Omega_n$ the set of $x\in \setR^d$ whose distance from $\supp f_n$ is smaller than $\lambda_n/2$. 
	Observe that
	\begin{align*} 	\int_{\mathbb{R}^d} & \int_{\mathbb{R}^d}\frac{|\sum_n f_n(x+h)-\sum_n f_n(x)|^p}{|h|^{d+sp}}\di x\di h\\
	\geq & \limsup_{N\to\infty} \sum_{n=1}^N \int_{B_{\lambda_n/2}}\int_{\bar \Omega_n}\frac{| f_n(x+h)- f_n(x)|^p}{|h|^{d+sp}}\di x\di h\\
	=& \limsup_{N\to\infty} \sum_{n=1}^N\Big( \int_{\setR^d}\int_{\setR^d}\frac{|f_n(x+h)-f_n(x)|^p}{|h|^{d+sp}}\di x\di h\\
	&\qquad\qquad\qquad-\int_{\mathbb{R}^d\setminus B_{\lambda_n/2}}\int_{\mathbb{R}^d}\frac{|f_n(x+h)-f_n(x)|^p}{|h|^{d+sp}} \di x\di h\Big).
	\end{align*}
	On the other hand
	\begin{align*}
	&\int_{\mathbb{R}^d\setminus B_{\lambda_n/2}}\int_{\mathbb{R}^d}\frac{|f_n(x+h)-f_n(x)|^p}{|h|^{d+sp}} \di x\di h
	\\
	&\le 2^{p}\norm{f_n}_{L^p}^p\int_{\mathbb{R}^d\setminus B_{\lambda_n/2}}\frac{1}{|h|^{d+sp}} \di h \le \frac{c(d)2^{p}}{sp} \left(\frac{2}{\lambda_n}\right)^{sp}\norm{f_n}_{L^p}^p.
	\end{align*}
	Combining these inequalities we get the sought conclusion.
\end{proof}

\begin{remark} The inequality \eqref{essumf} in the case $p=2$ was proven in \cite{AlbertiCrippaMazzuccato18}.
\end{remark}

Let us start with the construction of our examples. Let $p\ge 1$ be fixed.
We consider $v$ and $\rho$ as in \autoref{alberti,Crippa,Mazzucato}, and a family of disjoint open cubes $\set{Q_n}_{n\in \setN}$ contained in $B_{1/4}$. 
Assuming that the cube $Q_n$ has side of length $3\lambda_n$ and center at $x_n\in B_{1/4}$, we set 
\begin{align*}
	v_n(t,x):=\frac{\lambda_n}{\tau_n}v\left(\frac{t}{\tau_n},\frac{x-x_n}{\lambda_n}\right),
	\qquad
	\rho_n(t,x):=\gamma_n\rho\left(\frac{t}{\tau_n},\frac{x-x_n}{\lambda_n}\right),
\end{align*}
for every $x\in \setR^d$, $t\ge 0$ and $n\in \setN$ where $\gamma_n,\tau_n,\lambda_n\in (0,1/10)$ convergence to $0$ and 	\begin{equation*}
	\sup_n\lambda_n/\tau_n<\infty.
\end{equation*}
Observe that $u_n$ is supported in $Q_n$ and $\text{dist}(\supp u_n, \setR^d\setminus Q_n)\geq \lambda_n$ for every $n\in \setN$.
Let us set
\begin{equation*}
	b(t,x):=\sum_n v_n(t,x),\qquad u(t,x):=\sum_n \rho_n(t,x)
	\quad \qquad \forall x\in \setR^d,\ \quad\forall t>0.
\end{equation*}
Note that $\supp b_t\subset B_{1/2}$ and $\supp u_t\subset B_{1/2}$.
It remains to choose properly the parameters $\lambda_n$, $\gamma_n$, $\tau_n$ in order to get our three examples.

\begin{proof}[Proof of \autoref{th: example1}]
     We  choose
	\begin{equation*}
	\gamma_n=\exp\left\lbrace-\log(n)\log\log\log\log(n)\right\rbrace,\
	\tau_n=\frac{1}{\log(n)\log\log\log(n)},\ \lambda_n=\frac{1}{n^{1/d}\exp\left\lbrace\frac{\log(n)}{\log\log\log(n)}\right\rbrace}
	\end{equation*}
	for $n\in \setN$ big enough\footnote{It suffices that $\log\log\log\log(n)\geq 1$, namely $n\geq \exp\exp\exp\exp(1)$}. It is easily seen that $\sup_n\gamma_n\lambda_n^{-m}+\gamma_n<\infty$ for any $m$, it implies $u_0\in C_c^\infty(B_{1/2})$ and $u\in L^\infty((0,\infty)\times \setR^d)$.
	Let us check \eqref{ine1}.
	\begin{align*}
	\int_{B_1(0)}&\left(\exp\left\lbrace\frac{|\nabla b(t,x)|}{\log(1+|\nabla b(t,x)|)^a}\right\rbrace-1 \right)\di x\\
	&=\sum_n \int_{B_1(0)}\left(\exp\left\lbrace\frac{|\nabla v_n(t,x)|}{\log(1+|\nabla v_n(t,x)|)^a}\right\rbrace-1 \right)\di x\\
	&=\sum_n \int_{B_1(0)}\left(\exp\left\lbrace \frac{\frac{1}{\tau_n}|\nabla v(t/\tau_n,x)|}{\log(1+\frac{1}{\tau_n}|\nabla v(t/\tau_n,x)|)^a}\right\rbrace-1 \right\rbrace\lambda_n^d  \di x\\
	&\leq C\sum_{n}\left(\exp\left\lbrace\frac{\frac{1}{\tau_n}C}{\log(1+\frac{C}{\tau_n})^a}\right\rbrace-1 \right)\lambda_n^d\\
	&\leq \sum_{n}C\exp\left\lbrace C'|\log(\tau_n)|^{-a}\tau_n^{-1}\right\rbrace\lambda_n^d
	\\ &\leq C \sum_{n}\exp\left\lbrace C''\frac{\log(n)\log\log\log(n)}{(\log\log(n))^a}\right\rbrace\frac{1}{n\exp\left\lbrace \frac{d\log(n)}{\log\log\log(n)}\right\rbrace},
	\end{align*}
	where $C$, $C'$ and $C''$ are positive constants. Observe that for $n_0\in \setN$ large enough, we have
	\begin{equation*}
	C''\frac{\log\log\log(n)}{(\log\log(n))^a}\leq \frac{d}{2}\frac{1}{\log\log\log(n)}
	\qquad
	\forall n\ge n_0.
	\end{equation*}
	Thus, 
	\begin{align*}
	\int_{B_1(0)}\left(\exp\left\lbrace \frac{|\nabla b(t,x)|}{\log(1+|\nabla b(t,x)|)^a}\right\rbrace-1 \right)dx\leq C \sum_{n}\frac{1}{n\exp\left\lbrace\frac{d \log(n)}{2\log\log\log(n)}\right\rbrace}\lesssim\sum_n\frac{1}{n^2}<\infty.
	\end{align*}
	Let us eventually verify
    \begin{equation*}
    [u_t]_{W^{s,p}}=\infty
    \qquad
    \forall t>0, \ \forall s>0, \ \forall p\ge 1.
    \end{equation*}
	An application of \autoref{subesti} leads to
	\begin{align*}
	[u_t]_{W^{s,p}}^{p}&\geq \limsup_{N\to\infty} \sum_{n=1}^N \left([\rho_n(t,\cdot)]_{W^{s,p}}^p-\frac{c(d)2^{p}}{sp} \left(\frac{2}{\lambda_n}\right)^{sp}\norm{\rho_n(t,\cdot )}_{L^p}^p\right)\\
	&\geq 
	\limsup_{N\to\infty} \sum_{n=1}^N \gamma_n^p\lambda_n^{d-sp}\left([\rho(t/\tau_n,\cdot)]_{W^{s,p}}^p-C(s,p,d)\norm{\rho(t/\tau_n,\cdot )}_{L^p}^p\right)\\
	&\geq 
	\limsup_{N\to\infty} \sum_{n=1}^N \gamma_n^p\lambda_n^{d-sp}\left( C\exp\left\lbrace cpst/\tau_n\right\rbrace-C(s,p,d)\norm{\rho_0}_{L^p}^p\right),
	\end{align*}
	where in the last passage we used \autoref{alberti,Crippa,Mazzucato}(iv).
	Now observe that
	\begin{align*}
	\gamma_n^{p}& \lambda_n^{d-sp}\exp\left\lbrace\frac{cpst}{2\tau_n}\right\rbrace\\
	&=\frac{1}{n^{1-sp/d}}\exp\left\lbrace \log(n)\left(\frac{cpst}{2}\log\log\log(n)-p\log\log\log\log(n)-\frac{(d-sp)}{\log\log\log(n)}\right)\right\rbrace\\
	&\geq \frac{1}{n^{1-sp/d}}\exp\left\lbrace 2\log(n)\right\rbrace
	=n^{1+sp/d}\geq 1,
	\end{align*}
	for any $n$ large enough.
	Moreover $\sum_{n=1}^\infty \gamma_n^p\lambda_n^{d-sp}<\infty$, so 
	\begin{equation*}
	[u_t]_{W^{s,p}}^p\geq 
	C'\sum_{n\geq n_0}^\infty \exp\left\lbrace\frac{cpst}{2\tau_n}\right\rbrace-C''=\infty.
	\end{equation*}
	The proof is complete.
\end{proof}
Let us now pass to the second example.
	
\begin{proof}[Proof of \autoref{th: example 2}]
   For any $m\in \setN$ positive and any $\lambda>0$,  we  choose
	\begin{equation*}
	\tau_n=\frac{2c_0\lambda}{\log(n)\log\log\log(n)},\quad 
	\lambda_n=\frac{1}{\exp\left\lbrace \log(n)\log\log\log(n)\right\rbrace },\quad
	\gamma_n=\lambda_n^m,
	\end{equation*}
	for $n\geq 10^{100}$, where $c_0:=\norm{v}_{L^\infty}$. 
	Since $\sup_n\gamma_n\lambda_n^{-m}+\gamma_n<\infty$, we have $u_0\in C_c^m(B_{1/2})$ and $u\in L^\infty((0,\infty)\times\setR^d)$. 
	Let us check \eqref{zz1}. Using the identity $\lambda_n=\exp(-2c_0\lambda/\tau_n)$ we get 
	\begin{align*}
	\int_{B_1}  \left(\exp\left\lbrace \lambda|\nabla b_t(x)|\right\rbrace-1 \right)\di x &=\sum_n \int_{B_1}\left(\exp\left\lbrace \lambda |\nabla v_n(t,x)|\right\rbrace-1 \right) \di x\\
	\\&\leq C\sum_{n}\exp\left\lbrace c_0\lambda/\tau_n\right\rbrace\lambda_n^d\\
	&= C\sum_{n}\lambda_n^{d-\frac{1}{2}}<\infty.
	\end{align*}
	Let us now prove (i).
	\begin{equation*}
	\norm{u_t}_{W^{1,1}} =\sum_{n}\norm{\rho_n(t,\cdot )}_{W^{1,1}}
	=\sum_{n}\gamma_n\lambda_n^{d-1}\norm{\rho(t/\tau_n,\cdot )}_{W^{1,1}}
	\geq C \sum_n \gamma_n\lambda_n^{d-1}\exp\left\lbrace ct/\tau_n\right\rbrace
	\end{equation*}
	where in the last line we used \autoref{alberti,Crippa,Mazzucato}(iv).
	Now observe that
	\begin{equation*}
		 C \sum_n \gamma_n\lambda_n^{d-1}\exp\left\lbrace ct/\tau_n\right\rbrace
		\ge  C \sum_{n}\exp\left\lbrace\left(-2c_0\lambda(d+m-1)+ct\right)/\tau_n\right\rbrace=+\infty
	\end{equation*}
	provided $-2c_0\lambda(d+m-1)+ct\geq 0$, that is to say
	\begin{equation*}
	t\geq c_1\lambda
	\qquad
	\text{with}
	\quad
	c_1=2c_0(d+m-1)/c.
	\end{equation*}
	Let us finally prove (ii). For any $0<s<1$ we have
	\begin{align*}
	[u_t]_{W^{s,1}} & \geq \limsup_{N\to\infty} \sum_{n=1}^N \left(\norm{\rho_n(t,\cdot )}_{W^{s,1}}-\frac{c(d)2}{s} \left(\frac{2}{\lambda_n}\right)^{s}\norm{\rho_n(t,\cdot )}_{L^1}\right)\\
	&\geq 
	\limsup_{N\to\infty} \sum_{n=1}^N \gamma_n\lambda_n^{d-s}\left(\norm{\rho(t/\tau_n,\cdot )}_{W^{s,1}}-C(s,d)\norm{\rho(t/\tau_n,\cdot )}_{L^1}\right)
	\\
	&\geq 
	C \sum_{n=1}^\infty \gamma_n\lambda_n^{d-s}\exp\left\lbrace cst/\tau_n\right\rbrace-C'
	\\&= 
	C\sum_{n}\exp\left\lbrace\left(-2c_0\lambda(d+m-s)+cst\right)/\tau_n\right\rbrace-C',
	\end{align*}
	where we used $\sum_{n=1}^\infty \gamma_n\lambda_n^{d-s}<\infty$.
	From the previous estimate we deduce
	\begin{equation*}
	[u_t]_{W^{s,1}}=\infty
	\qquad
	\text{provided}
	\quad
	-2c_0\lambda(d+m-s)+cst\geq 0,
	\end{equation*}
	that implies our conclusion with $c_2=2c_0(d-1+m)/c$.
\end{proof}
\begin{proof}[Proof of \autoref{th: example3}]
Let us choose
	\begin{equation}
	\tau_n=\frac{1}{\log\log(n)},
	\qquad
	\lambda_n=\frac{1}{n^2\Phi(c_0\log\log(n))^{1/d}},
	\qquad
	\gamma_n=\frac{\lambda_n}{\log\log(n)}
	\qquad\text{for}\ n\geq 10^{100},
	\end{equation}
	where $c_0=\norm{\nabla v}_{L^\infty}$.
	First of all let us observe that $u_0\in C^1_c(B_{1/2})$, since $\gamma_n/\lambda_n=1/\log\log(n)$. In order to check \eqref{zz2},  we estimate
	\begin{equation*}
	\int_{\mathbb{R}^d}\Phi(|\nabla b_t(x)|)\di x
	=\sum_n \lambda^d_n \int_{B_2}\Phi(|\nabla v(t/\tau_n,x)|/\tau_n)\di x
	 = C\sum_{n}\frac{1}{n^{2d}}<\infty.
	\end{equation*}
	Moreover, using \autoref{alberti,Crippa,Mazzucato}(ii) we have
	\begin{align*}
	\norm{\nabla u_t}_{L^\infty}
	&\geq \sup_n \frac{\gamma_n}{\lambda_n} \norm{\nabla \rho(t/\tau_n, \cdot )}_{L^\infty}\\
	&\geq C \sup_n \frac{\gamma_n}{\lambda_n} \exp\left\lbrace ct/\tau_n\right\rbrace\\
	&=\sup_n\frac{C}{\log\log n}\exp\left\lbrace ct\log\log(n)\right\rbrace =\infty.
	\end{align*}
	The proof is complete.
\end{proof}

\section{Application to the 2D Euler equation}\label{section:Euler}
In this section we present an application of \autoref{th: Main 1} to the study of the 2D Euler equation with bounded initial vorticity in the class $L^{\infty}\cap W^{\alpha,p}$. We prove a propagation of regularity result that generalizes \cite[Corollary 1.1]{BahouriChemin94}.

Let us start by introducing the Cauchy problem associated to the 2D Euler equation in vorticity formulation. Here we set the problem in the 2 dimensional torus:
\begin{equation}\label{EU}
	\begin{dcases}
	\partial_t \omega_t+\div(b_t \omega_t)=0,\\
	b_t=K\ast \omega_t,\\
	\omega_0=\bar \omega,
	\end{dcases}\tag{E}
\end{equation}
where $\bar \omega$ is the initial data and $K$ is the Biot-Savart kernel. 

In this section, we consider only solutions of class $L^{\infty}([0,T]\times\T^2)\cap C([0,T];L^1(\T^2))$ satisfying the weak formulation
\begin{equation}\label{euler weak}
	\int_{\T^2} \omega_t \phi \di x-\int_{\T^2} \omega_0 \phi \di x=\int_0^t\int \omega_s\ b_s\cdot \nabla \phi \di x\di s
	\qquad
	\forall \phi\in C^{\infty}(\T^2).
\end{equation}
It is well-known since the work \cite{Yudovich63} that in this class \eqref{EU} admits a unique solution (in the sense of \eqref{euler weak}). We refer to \cite{BertozziMajda02,Chemin98,Lions96,MarchioroPulvirenti} for a detailed description of the classical theory for the 2D Euler equation.

Let us state the main result of this section.
\begin{theorem}\label{th:main2} 
	Let $0<\alpha\le 1$ and $p\ge 1$ be fixed. Consider a weak solution of \eqref{EU} $\omega\in L^{\infty}([0,T]\times\T^2)\cap C([0,T];L^1(\T^2))$. The following hold true:
	\begin{itemize}
		\item[(i)] if $\omega_0\in W^{\alpha,p}(\T^2)$ then
		\begin{equation}\label{euler1}
		\omega_t\in W^{\alpha_t,p}(\T^2)
		\qquad
		\text{with}
		\quad
		\alpha_t:=\alpha\  \frac{1}{1+C\norm{\omega_0}_{L^{\infty}}\alpha p t}
		\qquad
		\forall t\in [0,T],
		\end{equation}
		where $C>0$ is a universal constant;
		\item[(ii)] if $\omega_0\in C(\T^2)\cap W^{\alpha,p}(\T^2)$ with $p>1$ then
		\begin{equation}\label{euler2}
		\omega_t\in W^{\alpha',p}(\T^2)
		\qquad
		\text{for any}\ 0<\alpha'<\alpha, \quad \forall t\in [0,T].
		\end{equation}
		When $\alpha=1$ we also have
		\begin{equation}\label{euler2.5}
		\omega_t\in W^{1,p'}(\T^2)
		\qquad
		\text{for any}\ 1\le p'<p, \quad \forall\ t\in [0,T];
		\end{equation}
		\item[(iii)] If $\omega_0\in W^{\alpha,p}(\T^2)$ with $p>2/\alpha$ it holds $\omega_t\in  W^{\alpha,p}(\T^2)$ for any $t\in [0,T]$.
	\end{itemize}	
\end{theorem}

\begin{remark}
	The conclusion (iii) in \autoref{th:main2} can be also obtained using the H\"older theory for the Euler equation (see for instance \cite{BertozziMajda02}). Indeed, the Sobolev embedding gives $W^{\alpha,p}(\T^2)\subset C^{1-\frac{2}{p\alpha}}(\T^2)$  for $\alpha p>2$ that implies in turn $b\in L^1([0,T];C^{2-\frac{2}{p\alpha}}(\T^2))$ and the classical Cauchy-Lipschitz theory can be applied.
\end{remark}

Before proving \autoref{th:main2} let us recall the main properties of the Biot-Savart kernel $K$ in \eqref{EU}. In the whole space $\setR^2$ it can be explicitly written as
\begin{equation*}
	K(y):=\frac{1}{2\pi }\frac{y^{\perp}}{|y|^2},
\end{equation*}
while in our periodic setting has a more complicated form \footnote{For any function $f$ in $\T^2$, it holds
\begin{equation*}
K\ast f(x)=\sum_{n\in \mathbb{Z}}\int_{[0,1]^2}	K(x-y+n)f(y)\di y.
\end{equation*}} but still satisfies the following properties:
\begin{itemize}
	\item[(i)] $K\in L^1(\T^2;\setR^2)$;
	
	\item[(ii)] $\nabla K: \T^2\to \setR^{2\times 2}$ is a vector valued Calderon-Zygmund kernel, in particular there exists a constant $C>1$ such that
	\begin{equation}\label{eq:BMOestimate}
		\int_{\T^2} \exp\left\lbrace \frac{|\nabla (K\ast f)|}{C\norm{f}_{L^{\infty}}} \right\rbrace \di x \le C
		\qquad
		\forall f\in L^{\infty}(\T^2).
	\end{equation}
\end{itemize}
\eqref{eq:BMOestimate} follows from the fact that Calderon-Zygmund operators map $L^{\infty}$ in $\text{BMO}$ (see \cite{Stein}) and from the exponential integrability of $\text{BMO}$ functions (see \cite{Nieremberg61}).
We refer to \cite{Schochet96} for a detailed analysis of the Biot-Savart Kernel $K$ in the periodic setting.

The next lemma shows that \eqref{eq:BMOestimate} can be slightly improved when $f\in C(\T^2)$. It is basically a consequence of the fact that Calderon-Zygmund operator map $C(\T^2)$ to $\text{VMO}(\T^2)$ (the space of vanishing mean oscillation functions), see \cite[page 180]{Stein2}. 
\begin{lemma}\label{lemma:VMOestimate}
	Let $K$ as above. Then for any $f\in C(\T^2)$ and every $\beta>0$ it holds
	\begin{equation}\label{z20}
		\int_{\T^2} \exp\set{\beta |\nabla (K\ast f)|}\di x<C(\beta,\norm{K}_{L^1},\norm{f}_{L^{\infty}},\rho_f^{-1}(C/\beta)),\footnote{\text{This constant blows up when $\beta\to \infty$}}
	\end{equation}
	where $\rho_f (r):=\sup\set{|f(x)-f(y)|:\ x,y\in \T^2\ \text{with}\ \dist(x,y)\le r}$ is the modulus of continuity of $f$.
\end{lemma}
\begin{proof}
	Let us fix $\beta>1$ and $f\in C(\T^2)$. There exists $\bar f\in C^{\infty}(\T^2)$ such that $\norm{f-\bar f}_{L^{\infty}}\le \frac{C}{\beta}$ and $\norm{\nabla \bar{f}}_{L^{\infty}}\le \frac{2\norm{f}_{\infty}}{\rho_f^{-1}(C/\beta)}$ where $C$ is as in \eqref{eq:BMOestimate}. Observe that, since $K\in L^1(\T^2;\setR^2)$, we have 
	\begin{equation}
		\norm{\nabla(K\ast \bar f)}_{L^{\infty}}=\norm{K\ast \nabla \bar f}_{L^{\infty}}\le\frac{2\norm{f}_{L^\infty}\norm{K}_{L^1}}{\rho_f^{-1}(C/\beta)}.
	\end{equation}
	Therefore we can estimate
	\begin{align*}
			\int_{\T^2} \exp\set{\beta |\nabla (K\ast f)|}\di x \le &
			\int_{\T^2} \exp\set{\beta |\nabla (K\ast \bar f)|+\beta|\nabla (K\ast (f-\bar f))|}\di x\\
			\le & \exp\left\{\frac{2\beta\norm{f}_{L^\infty}\norm{K}_{L^1}}{\rho_f^{-1}(C/\beta)}\right\} 	\int_{\T^2} \exp\set{\beta|\nabla (K\ast(f-\bar f))|}\di x\\
			\le  & \exp\left\{\frac{2\beta\norm{f}_{L^\infty}\norm{K}_{L^1}}{\rho_f^{-1}(C/\beta)}\right\} \int_{\T^2} \exp\left\lbrace \frac{|\nabla (K\ast (f-\bar f))|}{C\norm{f-\bar f}_{L^{\infty}}} \right\rbrace \di x\\
			\le & C\exp\left\{\frac{2\beta\norm{f}_{L^\infty}\norm{K}_{L^1}}{\rho_f^{-1}(C/\beta)}\right\}<\infty,
	\end{align*}
	where in the last line we used \eqref{eq:BMOestimate}.
\end{proof}

\begin{proof}[Proof of \autoref{th:main2}]
	Any weak solution $\omega_t$ of \eqref{EU} is also a distributional solution of \eqref{CE} with drift $b_t:=K\ast \omega_t$. Observe that $\div b_t=0$ and it holds
	\begin{equation*}
    \sup_{t>0}	\int_{\T^2} \exp\left\lbrace \frac{|\nabla b_t |}{C\norm{\omega_0}_{L^{\infty}}} \right\rbrace \di x \le C,
	\end{equation*}
	thanks to \eqref{eq:BMOestimate} and the identity $\norm{\omega_t}_{L^{\infty}}=\norm{\omega_0}_{L^{\infty}}$. Applying \autoref{th: Main 1}(i) the first conclusion follows.
	
	Let us now address (ii). First of all observe that if $\omega_0\in C(\T^2)$ then $\omega_t\in C(\T^2)$ for any $t\ge 0$. It follows from the identity $\omega_t(x)=\omega_0((X_t)^{-1}(x))$ where $X$ is the flow associated to $b_t$, that is continuous together with its inverse thanks to \eqref{eq: Holder flow}. Therefore \autoref{lemma:VMOestimate} infers that, for any $\beta>0$ and $t\ge 0$
	\begin{equation}\label{z21}
	\int_{\T^2} \exp\set{\beta |\nabla b_t(x)|}\di x<\infty.
	\end{equation}
	Exploiting \autoref{lemma:VMOestimate} and the fact that the modulus of continuity of $\omega_t$ fulfills $\sup_{0<t<T}\rho_{\omega_t}\le \rho$ for some nondecreasing $\rho:(0,\infty)\to (0,\infty)$ satisfying $\lim_{r\to 0} \rho(r)=0$ we can strengthen \eqref{z21} as
	\begin{equation*}
	\sup_{0<t<T}\int_{\T^2} \exp\set{\beta |\nabla b_t(x)|}\di x<\infty
	\qquad
	\text{for any}\ \beta>0.
	\end{equation*}
	We are in position to apply \autoref{cor: main1} and conclude the proof of (ii). We eventually prove (iii). Since $W^{\alpha,p}(\T^2)\subset C(\T^2)$ when $p>2/\alpha$ we use (ii) to deduce that $\omega_t\in W^{\alpha',p}(\T^2)$ for any $0<\alpha<\alpha'$ and $t\ge 0$. This infers that $\nabla b_t \in W^{\alpha',p}(\T^2;\setR^{2\times 2})$ for any $0<\alpha'<\alpha$. Sobolev's embedding theorem implies that $b_t$ is a Lipschitz vector field with respect to the spatial variable. It is also clear that this estimate is locally uniform in time, therefore the standard Cauchy-Lipschitz theory applies and we conclude that $X_t$ is actually biLipschitz, thus $\omega_t=\omega_0(X_t^{-1})\in W^{\alpha,p}(\T^2)$.
\end{proof}

\appendix
\section{Appendix}
In this appendix we collect two technical results concerning functions whose gradient is exponentially integrable. 
The first one is a consequence of \cite[Theorem 8.40]{AdamsFournier75}, we add its proof for the reader convenience.
\begin{lemma}\label{lemma: log-lipschitz regularity}
	Let $f\in W^{1,1}(\T^d)$ satisfy
	\begin{equation*}
	\int_{\T^d}\exp\set{\beta |\nabla f(x)|} \di x=:K<\infty,
	\qquad
	\text{for some}\ \beta>0.
	\end{equation*}
	Then $f$ admits a continuous representative satisfying
	\begin{equation}\label{eq: log-lipschitz estimate}
	|f(x)-f(y)|\le \frac{C}{\beta}\dist(x,y)\log\left( \frac{CK}{\dist(x,y)^d} \right)
	\qquad
	\forall x,y\in \T^d,
	\end{equation}
	where  $C=C_d>2$.
\end{lemma}
\begin{proof}
	Thanks to the Poincar\'e inequality there exists a constant $\bar C>0$ such that
	\begin{equation}\label{z10}
	\dashint_{B_{2r}(x)}\abs{f(y)-\dashint_{B_r(x)}f }\di y \le \bar C r \dashint_{B_{2r}(x)} |\nabla f(y)|\di y \le \bar Cr \left( \dashint_{B_{2r}(x)} |\nabla f(y)|^p\di y\right)^{1/p},
	\end{equation}
	for any $x\in \T^d$, for any $r\in (0,2)$ and $p\ge 1$. Using the notation $f_{x,r}:=\dashint_{B_r(x)} f$ we deduce from \eqref{z10}
	\begin{equation*}
	\frac{\left(  \frac{\beta}{r} \dashint_{B_{2r}(x)}|f-f_{x,r}|\di y\right)^p}{p!}
	\le \frac{\bar C}{r^d\omega_d} \int_{\T^d} \frac{(\beta|\nabla f(y)|)^p}{p!} \di y,
	\qquad
	\forall p\in \setN,\ p\ge 1.
	\end{equation*}
	Taking the sum for $p\ge 1$ we conclude
	\begin{equation*}
	\exp\left(\frac{\beta}{r}\dashint_{B_{2r}(x)}|f-f_{x,r}|\di y\right)-1\leq  \frac{\bar C}{r^d\omega_d} (K-1).
	\end{equation*}
	This gives
	\begin{equation}\label{z12}
	|f_{x,2r}-f_{x,r}|\le \dashint_{B_{2r}(x)}|f-f_{x,r}|\di y\le \frac{r}{\beta}\log\left(\frac{K\bar C }{\omega_d r^d}\right),
	\end{equation}
	up to increase $\bar C$. Thanks to Morrey's inequality (see \cite{AdamsFournier75}) we know that $f\in C(\T^d)$ and thus $f(x)=\lim_{r\to 0}\dashint_{B_r(x)} f\di y$. Using this, \eqref{z12} and a standard iteration procedure we get
	\begin{equation*}
		|f_{x,r}-f(x)|\le \sum_{k=0}^{\infty}|f_{x,2^{-k}r}-f_{x,2^{-k-1}r}|
		              \le \sum_{k=0}^{\infty} \frac{r2^{-k}}{\beta}\log\left(\frac{K\bar C }{\omega_d (2^{-k}r)^d}\right)
		              \le \frac{C'r}{\beta}\log\left(\frac{K\bar C }{\omega_d r^d}\right),
	\end{equation*}
	for some $C'>0$ depending only on $d$.
	Observe that when $r=\dist(x,y)$ we have $|f_{x,r}-f_{y,r}|\le 2^d\dashint_{B_{2r}(x)}|f-f_{x,r}|\di y$, therefore
	\begin{align*}
		|f(x)-f(y)|\le & |f(x)-f_{x,r}|+2^d\dashint_{B_{2r}(x)}|f-f_{x,r}|\di y+|f_{y,r}-f(y)|\\
		           \le & \frac{(2C'+2^d)r}{\beta}\log\left(\frac{K\bar C }{\omega_d r^d}\right)
		           \le \frac{C}{\beta}\dist(x,y)\log\left( \frac{CK}{\dist(x,y)^d} \right),
	\end{align*}
	where $C=C_d$.
\end{proof}

\begin{lemma}\label{lemma:exponential lusin}
	Let $f\in W^{1,1}(\T^d)$ satisfy
	\begin{equation*}
	\int_{\T^d}\exp\set{\beta |\nabla f(x)|} \di x<\infty,
	\qquad
	\text{for some}\ \beta>0.
	\end{equation*}
	Then for any $0<\beta'\le \beta$ we have
	\begin{equation*}
		\frac{\beta'}{C_1} \frac{|f(x)-f(y)|}{\dist(x,y)}
		\le 1+\log\left(c_d  M\left( \exp\left\{\beta'|\nabla f|\right\}\right)(x)\right)
		\qquad
		\forall x,y\in \setR^d,
	\end{equation*}
	where $c_d\ge 1$ and $C_1>0$ depend only on $d$.
\end{lemma}

\begin{proof}
  Let us fix $x,y\in \setR^d$. Using Morrey's inequality (see \cite{AdamsFournier75}) we have
  \begin{equation}\label{zzz1}
  	\frac{1}{p!}\left(  \frac{\beta'|f(x)-f(y)|}{C r}  \right)^p
  	\le \dashint_{B_r(x)} \frac{(\beta'|\nabla f|)^p}{p!}\di z
  	\qquad
  	\forall p>d\ \text{integer},
  \end{equation}
  where $C>0$ depends only on $d$ and $r:=2\dist(x,y)$. In order to make notation short let us set
  \begin{equation*}
  	A:=	\frac{\beta'}{2C} \frac{|f(x)-f(y)|}{\dist(x,y)}.
  \end{equation*}
  Taking the sum in \eqref{zzz1} for $p$ between $d+1$ and $\infty$ we end up with
  \begin{equation*}
  	e^A-1-A-\frac{A^2}{2!}-...-\frac{A^d}{d!}
  	\le \dashint_{B_r(x)} \exp \left\{\beta'|\nabla f|\right\} \di z
  	\le M\left( \exp \left\{\beta'|\nabla f|\right\}\right)(x).
  \end{equation*}
  Now observe that, when $A\ge 1$, we have 
  \begin{equation*}
  	e^{A/2} \le c_d( e^A-1-A-\frac{A^2}{2!}-...-\frac{A^d}{d!})\le c_dM\left( \exp \left\{\beta'|\nabla f|\right\}\right)(x),
  \end{equation*}
  for some constant $c_d\ge 1$, thus we deduce
  \begin{equation*}
  	\frac{A}{2}\le \log\left( c_dM\left( \exp\left\{\beta'|\nabla f|\right\}-1\right)(x)   \right)
  	\qquad
  	\text{when}\ A>1,
  \end{equation*}
  that trivially gives
  \begin{equation*}
  	\frac{A}{2} \le \frac{1}{2}+\log \left(c_dM \exp\left\{\beta'|\nabla f|\right\}(x) \right),
  \end{equation*}
  without any restriction on $A$. Recalling the definition of $A$ we conclude.
\end{proof}

\end{document}